\documentclass[10pt,a4paper]{article}

\usepackage{helvet}

\usepackage{amsmath,amsxtra,amsthm,amssymb,xr}
\usepackage[all]{xy}
\usepackage{lineno}

\newtheorem{theorem}{Theorem}[section]
\newtheorem{lemma}[theorem]{Lemma}

\newtheorem{conjecture}[theorem]{Conjecture}

\input cyracc.def       

\newcommand{\p}{\mathfrak{p}}
\newcommand{\Qbar}{\overline{\mathbb{Q}}}
\newcommand{\Tr}{\operatorname{Tr}}
\newcommand{\Z}{\mathbb{Z}}

\newcommand{\GL}{\operatorname{GL}}
\newcommand{\C}{\mathbb{C}}
\newcommand{\Gal}{\operatorname{Gal}}
\newcommand{\Q}{\mathbb{Q}}
\newcommand{\Ree}{\mathrm{Re}}

\newcommand{\comment}[1]{}

\title{Non-vanishing of Artin-twisted $L$-functions of Elliptic Curves}
\author{Thomas Ward}
\begin{document}
\date{}
\maketitle
\noindent{\bf Abstract:}
{\it Let $E$ be an elliptic curve and $\rho$ an Artin representation, both defined over $\Q$. Let $p$ be a prime at which $E$ has good reduction. We prove that there exists an infinite set of Dirichlet characters $\chi$, ramified only at $p$, such that the Artin-twisted $L$-values $L(E,\rho\otimes \chi,\beta)$ are non-zero when $\beta$ lies in a specified region in the critical strip (assuming the conjectural continuations and functional equations for these $L$-functions). The new contribution of our paper is that we may choose our characters to be ramified only at one prime, which may divide the conductor of $\rho$.}
\section{Introduction}
An Artin representation of $\Gal(\Qbar/\Q)$ is a continuous finite-dimensional complex representation which factors through a finite extension $K/\Q$.
If $E$ is an elliptic curve over $\Q$, there exists an $L$-series $L(E,\rho,s)$ associated to $E$ and $\rho$, which appears naturally as a factor of the $L$-function of $E$ over $K$.
\par
The function $L(E,\rho,s)$ is defined by an Euler product which converges only on the region $\Ree(s)>3/2$, but conjecturally it may be analytically continued to the entire complex plane and will satisfy a functional equation (we give further details in \S\ref{AFEsection}).
Assuming this continuation, we are interested in the vanishing properties of $L(E,\rho \otimes \chi,\beta)$ for Dirichlet twists $\chi$, where $\beta$ lies in the critical strip.
\par
Following the work of Langlands, there conjecturally exists a tempered cuspidal automorphic representation $\pi$ of $\GL_n(\mathbb{A}_\Q)$ (where $n=2\dim \rho$) such that $L(E,\rho,s) = L(\pi,s-1/2)$. The existence of such a $\pi$ immediately implies the claimed functional equation and analytic continuation.  
There are many results on the non-vanishing of automorphic $L$-functions in the literature: in particular, in \cite{Rohrlich2} Rohrlich shows the existence of infinitely many ray class characters $\chi$ such that $L(\pi \otimes \chi, \beta) \neq 0$ for any $\beta \in \C$, where $\pi$ is an irreducible cuspidal representation of $\GL_n(\mathbb{A}_F)$ with $n=1$ or $2$, for any number field $F$.
In \cite{Barthel-Ramakrishnan} Barthel and Ramakrishnan extend this result to all $n$ for $\Ree(\beta) \notin [1/n,1-1/n]$, or $\Ree(\beta) \notin [2/(n+1), 1-2/(n+1)]$ when $\pi$ is tempered. 
In \cite{Luo} Luo extends this further to the region $\Ree(\beta) \notin [2/n, 1-2/n]$ when the base field is $\Q$.
\par
These non-vanishing theorems are initially proved for twists which are unramified at the primes dividing the conductor of $\pi$. 
It is simple to relax this assumption: 
one can choose a character $\chi_0$ with no restriction on its ramification, and then apply the non-vanishing theorem to the automorphic representation $\pi \otimes \chi_0$ to obtain an infinite set of twists $\chi$ such that $L(\pi \otimes \chi \otimes \chi_0,\beta) \neq 0$. In this way we get a set of characters $\chi \otimes \chi_0$ which satisfy the desired non-vanishing property, but which can be arbitrarily ramified at the primes dividing the conductor $N_\pi$. However, this method cannot produce a set of twists ramified at only one prime, if that prime divides $N_\pi$.
The purpose of this paper is to prove a non-vanishing result for such a set of twists, for the case of the Artin-twisted $L$-function $L(E,\rho,s)$ over $\Q$.
The precise statement of our main theorem is as follows.
\begin{theorem}\label{thm1}
Let $p$ be a prime at which $E$ has good reduction. 
Suppose that the $L$-function $L(E,\rho\otimes\chi,s)$ satisfies the conjectural analytic continuation and functional equation for all Dirichlet twists $\chi$. Then, for any $\beta$ in the critical strip $\{ s \in \C : 1/2 \leq \Ree(s) \leq 3/2 \}$ satisfying
\[
\Ree(\beta) \; \notin \; \Big[ \frac{1}{2}+ \frac{2}{2 \dim \rho +1} \,, \; \frac{3}{2}- \frac{2}{2 \dim \rho +1} \, \Big]
\]
if $\dim \rho \geq 2$, and $\Ree(\beta) \neq 1$ if $\dim \rho = 1$,
we have $L(E,\rho\otimes\chi,\beta) \neq 0$
for infinitely many Dirichlet characters $\chi$ which are ramified only at $p$. 
\end{theorem}
We must point out that our theorem is weaker than that of Barthel and Ramakrishnan from \cite{Barthel-Ramakrishnan}: it holds for a more restricted class of tempered cuspidal automorphic representations and only for base field $\Q$.
Further, it holds on a smaller region than Luo's result from \cite{Luo}. Our new contribution is to allow the twists $\chi$ to be ramified at only a single prime, which may divide the conductor of $\rho$.
To prove Theorem \ref{thm1} we follow a standard technique of using the approximate functional equation and averaging over twists. In particular we follow the method of Luo from \cite{Luo} (previously used by Iwaniec in \cite{Iwaniec}). 
\par
We conclude the introduction with the following remark: suppose $\rho$ is 2-dimensional, irreducible and odd. 
Under these assumptions, except for certain icosahedral examples of $\rho$, it is known that $\rho$ is equivalent to the representation given by a weight 1 newform (this is proved by Buzzard, Dickinson, Shepherd-Barron and Taylor in \cite{Buzzard-et-al}). As we also know that $E$ is modular by the work of Wiles et al in \cite{Wiles} and \cite{Breuil-et-al} we may write $L(E,\rho,s)$ as a Rankin convolution of modular forms. For such an $L$-function, the functional equation and continuation are known (in fact, the existence of the automorphic representation $\pi$ of $\GL_4$ has been proved by Ramakrishnan in \cite{Ramakrishnan})
so Theorem \ref{thm1} holds unconditionally in this case.
\section{Conductors of twisted Artin representations}\label{ConductorsSection}
In our main theorem, we impose no restrictions on the ramification of $\rho$ at the prime $p$. For this reason we first prove a result on the Artin conductors of the twists $\rho\otimes\chi$; we begin with a preparatory lemma on ramification groups.
\begin{lemma}\label{irholemma}
Let $p$ be a prime. Let $L/\Q_p$ be a finite Galois extension of local fields, and suppose  
\[
\phi \;:\;\Gal(L/\Q_p) \rightarrow A
\]
is a homomorphism of groups. Let $K_n=\Q_p(\mu_{p^n})$ and regard $\rho$ as a homomorphism $\Gal(LK_n/\Q_p) \rightarrow A$ by extension through the quotient map.
Then there exists a fixed integer $i_\phi$ such that
\[
G_i(LK_n/\Q_p) \;\subseteq \; \ker \phi
\]
for all $i \geq i_\phi$, for all sufficiently large $n$. Here $G_i(LK_n/\Q_p)$ denotes the $i$-th ramification group of the extension $LK_n/\Q_p$.
\end{lemma}
\begin{proof}
Let $H=\Gal(LK_n/L)$, so that $H \subseteq \ker \phi$. Quoting \cite[Chapter II, \S10]{Neukirch} we have the formula
\[
\frac{G_i(LK_n/\Q_p) \,H}{H} \;=\; G_t(L/\Q_p)
\]
where $t=\eta_{LK_n/L}(i)$. Here, $\eta_{LK_n/L}$ is the function which defines the upper numbering for the ramification groups of $LK_n/L$ (see \cite{Neukirch}). Let $t'$ be the some integer such that $G_{t'}(L/\Q_p)=\{1\}$. 
We will show that there exists a fixed $i_\phi$ such that $\eta_{LK_n/L}(i) \geq t'$ when $i\geq i_\phi$, for all sufficiently large $n$. This will establish the claim, as the formula above implies $G_i(LK_n/\Q_p) \subseteq H \subseteq \ker \phi$ 
for such $i$.
\par
By definition, we have 
\[
\eta_{LK_n/L}(i) \;=\; \frac{1}{g_0} (g_1+g_2+\dots+g_i)
\]
where we write $g_i = |G_i (LK_n/L)|$. Enlarging $L$ if necessary, assume that $\Q_p(\mu_{p^\infty}) \cap L=K_m$. Suppose we have $p^s-1 \geq i \geq p^{s-1}$, where $n \geq s \geq m$. One checks that
\[
G_i (LK_n/L) \;\supseteq\; \Gal (LK_n/LK_s) 
\]
so for such $i$ we have $g_i \geq p^{n-s}$. We also observe that $g_i=p^{n-m}$ for $p^{m-1}\geq i \geq 0$. Therefore we have
\begin{eqnarray*}
\eta_{LK_n/L}(p^s-1)
& \geq &
\frac{1}{g_0} \sum_{i=1}^{p^{m-1}-1} g_0  \;+\; \frac{1}{g_0} \sum_{t=m}^s  p^{n-s} (p^s-p^{s-1})
 \\
 &=& p^{m-1}(1+(s-m)(p-1))-1.
\end{eqnarray*}
We may choose $s$ to make this quantity as large as we like, independently of $n$ (provided $n$ is sufficiently large), which completes the proof.
\end{proof}
Now let $\rho$ be an Artin representation of $\Gal(\Qbar/\Q)$. The global Artin conductor associated to $\rho$ may be given as a product
\[
N_\rho \;=\; \prod_{p \; \mathrm{prime}} p^{n_p(\rho)} 
\]
where the local conductors $n_p(\rho)$ are almost all zero (their precise definition is given below).
\begin{lemma}\label{newlemma}
We have
\[
n_p(\rho\otimes\chi) \; = \; n_p(\chi) \dim \rho
\]
when $n_p(\chi)$ is sufficiently large.
\end{lemma}
\begin{proof}
Fix a character $\chi$ of conductor $p^n$; we know that $\chi$ factors through $\Gal(K_n/\Q)$ where $K_n=\Q(\mu_{p^n})$. As $\rho$ is an Artin representation, $\rho$ factors through $\Gal(L/\Q)$ for some finite extension $L/\Q$. We may regard $\rho \otimes \chi$ as a representation of $\Gal(LK_n/\Q)$.
\par
Let $G$ be the decomposition group of $LK_n/\Q$ at $p$, and let $G_i =G_i((LK_n)_{\p_n}/\Q_p)$ be the $i$-th ramification group, where $\p_n$ denotes a prime of $LK_n$ above $p$. By definition of the Artin conductor (see \cite[Chapter VII]{Neukirch}) we have
\[
n_p(\rho \otimes \chi) \;=\; \sum_{i=0}^{\infty} \frac{g_i}{g_0} \left( \dim \rho -\frac{1}{g_i} \sum_{s \in G_i} \Tr \rho(s) \, \chi(\rho) \right) 
\]
where $g_i = |G_i|$. We observe that
\[
 \frac{1}{g_i} \sum_{s \in G_i} \chi(s) \Tr (\rho)
\]
is the inner product of the characters $\rho$ and $\overline{\chi}$.  This equals the number of copies of $\overline{\chi}$ in $\rho$, as representations restricted to $G_i$. 
Let $i_\rho$ be the smallest integer such that $G_i \subseteq \ker \rho$
for $i \geq i_\rho$. By Lemma \ref{irholemma} we know that $i_\rho$ is independent of $n$, so by taking $n$ sufficiently large we can assume that $\chi$ is non-trivial on $G_{i_\rho}$. Therefore if $i< i_\rho$, the sum above must be zero as $\rho$ restricted to $G_{i_\rho}$ is a sum of copies of the trivial representation, and $\chi$ is non-trivial. Also, if $i \geq i_\rho$ then $\Tr \rho (s) = \dim \rho$ for any $s \in G_i$ by the definition of $i_\rho$. Putting this into the formula above we deduce that $n_p(\rho\otimes\chi) \; = \; n_p(\chi) \dim \rho$.
\end{proof}
\section{The approximate functional equation}\label{AFEsection}
Let $E$ be an elliptic curve and $\rho$ an Artin representation, both defined over $\Q$. 
The $L$-series of $E$ twisted by $\rho$ is defined by setting
\[
L(E,\rho,s) \;=\; \prod_{q \; \mathrm{prime}} P_q(E,\rho,q^{-s})^{-1}
\]
where the local polynomial at the prime $q$ is given by
\[
P_q(E,\rho,T) \;=\; \det \Big( 1-\mathrm{Frob}_q^{-1} \cdot T \big| (H_l^1(E) \otimes_{\Qbar_l} V_{\rho,l} ) \Big).
\]
where $l$ is a prime different from $q$. Here, $H_l^1(E)$ is the dual of the $l$-adic Tate module of $E$, and $V_{\rho,l}$ is the representation obtained from $\rho$ by extending scalars to $\Qbar_l$.
This Euler product converges only on the right half-plane $\Ree(s)>3/2$ but conjecturally it may be continued to a holomorphic function on the whole complex plane.
Let us put
\[
L_\infty(s) \; = \;  \left( 2 (2 \pi)^{-s} \Gamma(s) \right)^{\dim \rho}  
\]
and define the completed $L$-function
\[
\widehat{L}(E,\rho,s) \; := \; L_\infty (s) \, L(E,\rho,s) .
\]
We also write $N(E,\rho)$ for the conductor associated to the twist of $E$ by $\rho$. Then, we have the following conjecture (see \cite{Tate} or \cite[\S2]{VladTim} for example):
\begin{conjecture}\label{feqn}
The completed $L$-function $\widehat{L}(E,\rho,s)$ has an analytic continuation to the whole complex plane, and satisfies the functional equation
\[
\widehat{L}(E,\rho,s) \;=\;  w(E,\rho) \, N(E,\rho)^{1-s} \, \widehat{L}(E,\rho^*,2-s)  
\]
where $\rho^*$ is the contragredient representation to $\rho$, and the root number $w(E,\rho)$ is a complex number of absolute value $1$.
\end{conjecture}
Suppose $\rho$ is an Artin representation for which Conjecture \ref{feqn} is satisfied. Let
\[
L(E,\rho,s)  \; = \; \sum_{n \geq 1} c_n n^{-s} 
\]
be the Dirichlet series expression which is valid for $\Ree(s)>3/2$.
For $\beta \in \C$ and positive $u \in \mathbb{R}$ we define the function 
\[
F_\beta (u)  \;:=\;  \frac{1}{2\pi i} \int_{\Ree(s)=2} L_\infty(s+\beta) \, u^{-s} \, \frac{ds}{s}.
\]
Now we assume that $1 \leq \Ree(\beta) \leq 3/2$.
Using Cauchy's theorem and the functional equation, we obtain the following formula (the approximate functional equation):
\begin{equation}\label{AFE}
\widehat{L}(E,\rho,\beta)  \;=\; \sum_{n \geq 1} c_n \, n^{-\beta} F_\beta (ny) \;+\; w  \, N^{1-\beta} \sum_{n \geq 1} c_n^* \, n^{\beta-2} F_{2-\beta} \left( \frac{n}{N y} \right) 
\end{equation}
for $y>0$, where $N=N(E,\rho)$ and $w=w(E,\rho)$. We also note that the function $F_\beta (u)$ has the following properties:
\[
u^{k} F_\beta (u) \rightarrow 0 \quad \mbox{as} \quad u \rightarrow \infty
\]
for any $k\geq 0$, and
\[
F_\beta (u) \rightarrow L_\infty(\beta) \quad \mbox{as} \quad u \rightarrow 0.
\]
We can check these by shifting the line of integration to $\Ree(s)=u^{k+\theta}$ and $\Ree(s)=u^{-\theta}$ respectively, for $\theta$ small and positive.
Further, one can show that 
\[
F_\beta (u) \;\ll\; 1 + u^\theta
\]
as $u$ approaches zero, for any small $\theta>0$. These statements also hold for $F_{2-\beta} (u)$. 
\section{Proof of Theorem \ref{thm1}}\label{AverageSection}
We now fix a prime $p$ at which $E$ has good reduction. We fix the Artin representation $\rho$ and let $\chi$ be a primitive Dirichlet character modulo $p^a$. Taking $a=n_p(\chi)$ to be sufficiently large, by Lemma \ref{newlemma} we know that $n_p(\rho\otimes\chi)= a \dim \rho$.
As we assumed $E$ has good reduction at $p$, we have
\[
N(E,\rho \otimes \chi) \;=\; N \, p^{2 a \dim \rho} .
\]
where $N$ is the prime-to-$p$ part of the conductor, which will remain fixed.
\par
We consider the twisted $L$-series
\[
L(E,\rho\otimes\chi,s) \;=\; \sum_{n \geq 1} \chi(n) \,c_n\, n^{-s}.
\]
In order to apply the approximate functional equation, we assume that it has the analytic continuation and functional equation specified in Conjecture \ref{feqn} for every Dirichlet character $\chi$ of conductor $p^a$ (which is hypothesised in Theorem \ref{thm1}).
\\\\
By the functional equation, it suffices to consider the values of $L(E,\rho\otimes\chi,s)$ in the right half of the critical strip. Let $\beta \in \C$ satisfy $1\leq \Ree(\beta)\leq 3/2$;
we apply equation \eqref{AFE} to write
\begin{eqnarray*}
\widehat{L}(E,\rho\otimes\chi,\beta)  &=& \sum_{n \geq 1} \chi(n) \,c_n \, n^{-\beta} F_\beta \left( \frac{ny}{N p^{2ad} } \right)
\\
&+& w(E,\rho,\chi) \, (Np^{2ad} )^{1-\beta} \; \sum_{n \geq 1} \overline{\chi}(n)\,c_n^* \, n^{\beta-2} F_{2-\beta} \left(\frac{n}{y}\right)
\end{eqnarray*}
for $y>0$; here we have re-normalised the variable $y$ via $y \mapsto yN^{-1}p^{-2ad}$. We observe that the $\chi$-twist does not change the Gamma-factor $L_\infty(s)$ so it does not affect the function $F$. 
\par
We now average over primitive characters modulo $p^a$: we define
\[
A(n) \;:=\; \sideset{}{^*}\sum_{\chi \; \mathrm{mod} \; p^a} \chi(n) \;F_\beta \left( \frac{ny}{N p^{2ad} } \right)
\]
and 
\[
B(n) \;:=\;   \sideset{}{^*}\sum_{\chi \; \mathrm{mod} \; p^a} w(E,\rho,\chi) \, \overline{\chi}(n) \, p^{2ad(1-\beta)} 
\]
where the symbol $\sum^*$ denotes the sum over primitive characters only.
Then we have
\begin{eqnarray*}
\sideset{}{^*}\sum_{\chi \; \mathrm{mod} \; p^a} \widehat{L}(E,\rho\otimes\chi,\beta)   &=& \sum_{n \geq 1} A(n) \,c_n \, n^{-\beta} 
\\
&+&  N^{1-\beta} \sum_{n \geq 1} B(n) \,c_n^* \, n^{\beta-2} F_{2-\beta} \left(\frac{n}{y}\right) .
\end{eqnarray*}
We will proceed to show that this sum tends to infinity with $a$ when 
\[
 \frac{3}{2}- \frac{2}{2d +1} \,<\, \Ree(\beta) \,\leq\, \frac{3}{2} 
\]
which will establish Theorem \ref{thm1}. We write the sum above in the form
\[
 \sideset{}{^*}\sum_{\chi \; \mathrm{mod} \; p^a} \widehat{L}(E,\rho\otimes\chi,\beta)
\;=\;
A(1) \;+\; \Sigma_1 \;+\; \Sigma_2
\]
where we have put
\[
\Sigma_1 \,=\, \sum_{n \geq 2} A(n) \,c_n \, n^{-\beta}
\quad \;\mbox{and}\; \quad 
\Sigma_2 \,=\, N^{1-\beta} \sum_{n \geq 1} B(n) \,c_n^* \, n^{\beta-2} F_{2-\beta} \left(\frac{n}{y}\right) .
\]
We will show that the parameter $y$ may be specified so that $|A(1)|\gg |\Sigma_1|$ and $|A(1)|\gg |\Sigma_2|$ as $a\rightarrow \infty$ which will prove the result.
\section{Estimating the sums}\label{EstimatingSums}
We introduce a parameter $x$ such that $xy = p^{2da}$ and assume both $x$ and $y$ are fixed positive powers of $p^a$ (we will discuss how they can be specified in \S\ref{Conclusion}). We have
\[
A(1) \;=\; \sideset{}{^*}\sum_{\chi \; \mathrm{mod} \; p^a} F_\beta \left( \frac{y}{N p^{2a d} } \right).
\]
By our choice of $x$ and $y$ we have
\[
 \frac{y}{N p^{2ad} } \;=\; \frac{1}{N x} \;\rightarrow\; 0
\]
as $a\rightarrow \infty$. Recalling that $F_\beta (u)$ tends to the non-zero constant $L_\infty(\beta)$ as $u\rightarrow 0$ we have
\[
|A(1)| \; \sim \,\; 
\sideset{}{^*}\sum_{\chi \; \mathrm{mod} \; p^a} 1 
\]
and therefore $|A(1)| \gg p^a$ (recalling that $p$ is constant and $a$ is tending to infinity). 
Next we consider $\Sigma_1$; we have
\[
|A(n)| \;=\;  
\bigg|  F_\beta \left( \frac{ny}{N p^{2ad} } \right) \sideset{}{^*}\sum_{\chi \; \mathrm{mod} \; p^a} \chi(n) \; \bigg| .
\]
Suppose first that $n\leq x^{1+\epsilon}$ for $\epsilon$ small and positive. 
Using the estimate $F_\beta (u) \;\ll\; 1 + u^\theta$, we obtain 
\[
|A(n)| \; \ll \;  
x^\epsilon \,\bigg|\,\; \sideset{}{^*}\sum_{\chi \; \mathrm{mod} \; p^a} \chi(n) \; \bigg| .
\] 
By basic properties of character sums we have
\[
\sideset{}{^*}\sum_{\chi \; \mathrm{mod} \; m} \chi(n) \;= \sum_{b|(n-1,m)} \varphi(b) \, \mu\left( \frac{m}{b}\right)
\]
if $(n,m)=1$ (see \cite[3.8]{Iwaniec-Kowalski}). From this we infer that
the character sum factor in $|A(n)|$ is $\ll p^a$ if $n-1$ is divisible by $p^{a-1}$, and zero otherwise. The Dirichlet coefficients $c_n$ are known to satisfy $|c_n| \leq n^{1/2 + \epsilon}$ for any $\epsilon >0$, and putting these facts together we get
\begin{equation}\label{estimate1}
\Big| \sum_{2 \leq n \leq x^{1+\epsilon}}  A(n) \,c_n \, n^{-\beta} \Big|  \;\ll\;  p^a \, x^\epsilon \sideset{}{^\dagger}\sum_{2 \leq n \leq x^{1+\epsilon}}  n^{1/2-\Ree(\beta) + \epsilon}
\end{equation}
where $\sum^\dagger$ denotes the sum over integers congruent to 1 mod $p^{a-1}$ only. We then observe that
\begin{align*}
\sideset{}{^\dagger}\sum_{2 \leq n \leq x^{1+\epsilon}}  n^{-\theta} 
& =
\; \sum_{1 \leq r \leq x^{1+\epsilon} p^{1-a}}  (1+rp^{a-1})^{-\theta} 
\\
& \ll \; p^{-a\theta} \; \big( x^{1+\epsilon} \; p^{1-a} \big)^{1-\theta+\epsilon} \;
\sum_{1 \leq r }  r^{-1-\epsilon} 
\\
& \ll \; p^{-a+\epsilon} \, x^{1-\theta+\epsilon}
\end{align*}
for any constant $\theta>0$. Putting $\theta=\Ree(\beta)-1/2-\epsilon$ and combining this with \eqref{estimate1} we deduce
\begin{eqnarray*}
\Big| \sum_{2 \leq n \leq x^{1+\epsilon}}  A(n) \,c_n \, n^{-\beta} \Big| &\ll&   x^{3/2- \Ree(\beta) +\epsilon} .
\end{eqnarray*}
We must deal with the terms for $n> x^{1+\epsilon}$; however the fact that $u^{k} F_\beta (u) \rightarrow 0$ as $u \rightarrow \infty$ for any $k\geq 0$ shows that the contribution of these terms is negligible. We conclude that
\begin{equation}\label{Sigma1bound}
|\Sigma_1| \;\ll\;  x^{3/2-\Ree(\beta)+\epsilon} .
\end{equation}
We now consider the sum
\[
\Sigma_2 = N^{1-\beta} \sum_{n \geq 1} B(n) \,c_n^* \, n^{\beta-2} F_{2-\beta} \left( \frac{n}{y} \right).
\]
As above, it will suffice to bound the terms with $n \leq y^{1+\epsilon}$ as the tail will become constant due to the rapid decay of $F_{2-\beta} (u)$ as $u \rightarrow \infty$.
By the estimates for $F_{2-\beta} (u)$ in \S\ref{AFEsection} we have the bound
\[
F_{2-\beta} (u) \;\ll\; \frac{1+u^\epsilon}{1+u^2}
\]
for any small $\epsilon>0$. Following Luo's proof from \cite{Luo}, we apply Cauchy's inequality and split up the product to get
\[
\Big| \sum_{1 \leq n \leq y^{1+\epsilon}} B(n) \,c_n^* \, n^{\beta-2}\, F_{2-\beta} \left(  \frac{n 
}{y} \right) \Big|
 \;\ll\;  
\sum_{1 \leq n \leq y^{1+\epsilon}} \big| B(n) \,c_n^* \big| \, n^{\Ree(\beta)-2} \, \frac{1+(\frac{n}{y})^\epsilon}{1+(\frac{n}{y})^2}
\]
\[
\ll  
\left( \sum_{n \leq y^{1+\epsilon}} |c_n^*|^2 n^{2(\Ree(\beta)-2)} \Big(1+\Big(\frac{n}{y}\Big)^\epsilon \Big) \right)^{1/2}
\!\!
\left( \sum_{n \leq y^{1+\epsilon}} |B(n)|^2 \Big(1+\Big(\frac{n}{y}\Big)^2\Big)^{-1} \right)^{1/2} 
\]
\[
\ll \; y^{\Ree(\beta)-1+\epsilon} \left( \sum_{n \leq y^{1+\epsilon}} |B(n)|^2 \Big(1+\Big(\frac{n}{y}\Big)^2\Big)^{-1} \right)^{1/2} .
\]
Here we have bounded the first factor in the product above by a method analogous to that used for $\Sigma_1$. Let us define 
\[
H(u) := \frac{1}{\pi (1+u^2)} 
\] 
which is the Fourier transform of $e^{-2\pi|u|}$. By the estimate above we have
\[
|\Sigma_2|\;\ll \; y^{\Ree(\beta)-1+\epsilon} \left( \sum_{n \in \Z} |B(n)|^2 H\left( \frac{n}{y} \right) \right)^{1/2} .
\]
Here we have increased the range of the sum in order to apply Poisson summation later. Recall that
\[
B(n) \;=\;  \sideset{}{^*}\sum_{\chi \; \mathrm{mod} \; p^a} w_\chi \, \overline{\chi}(n) \, p^{2ad(1-\beta)} ,
\]
where we have written $w_\chi$ for the root number $w(E,\rho\otimes\chi)$. We have
\[
\sum_{n \in \Z} |B(n)|^2  H\left( \frac{n}{y} \right) 
 \,\leq\, 
\sideset{}{^*}\sum_{\chi \; \mathrm{mod} \; p^a} \; \sideset{}{^*}\sum_{\psi \; \mathrm{mod} \; p^a} p^{4ad (1-\Ree(\beta))}
\bigg| w_\chi \, \overline{w}_\psi \sum_{n \in \Z} 
\overline{\chi}\psi(n)  \,H\left( \frac{n}{y} \right) \!\! \bigg|.
\]
First we consider the diagonal terms in this sum: those with $\chi=\psi$. For these terms we get
\[
\sideset{}{^*}\sum_{\chi \; \mathrm{mod} \; p^a}  p^{4ad(1-\Ree(\beta))} 
 \bigg| w_\chi \, w_{\overline{\chi}} 
 \sum_{n \in \Z}   H\left( \frac{n}{y} \right) \bigg|
 \;\ll\; 
 p^{4ad(1-\Ree(\beta)) +a}  \; \sum_{n \in \Z}  \, H\left( \frac{n}{y} \right) 
\]
recalling that $|w_\chi|=1$. By the Poisson summation formula we have
\[
\sum_{n \in \Z}   H\left( \frac{n}{y} \right) \;=\; y \sum_{h \in \Z}   T(yh)
\]
where $T(u)=e^{-2\pi |u|}$. All terms in this sum decay exponentially with $y$, except that for $h=0$. So $y \sum_{h \in \Z}   T(yh) \ll y$ which implies that the diagonal terms are 
\[
\ll \; y \; p^{4ad(1-\Ree(\beta)) +a}   .
\]
Now we consider the terms with $\chi \neq \psi$. The characters $\chi \overline{\psi}$ are not primitive in general, therefore we use Poisson summation in the following form:
\[
 \sum_{n \in \Z} \overline{\chi}\psi(n) \, H\left( \frac{n}{y} \right)
 \;=\;
 \frac{y}{p^a} \, \sum_{h \in \Z} \tau_h(\chi\overline{\psi})  \, T\left( \frac{yh}{p^a} \right).
\]
Here we write 
\[
\tau_n(\chi\overline{\psi})= \sum_{r \,\mathrm{mod}\, p^a} \chi\overline{\psi}(r) \, e^{2 \pi i nr/p^a}
\] 
for the discrete Fourier transform of the character $\chi\overline{\psi}$ which is defined modulo $p^a$. This may be difficult to estimate when $\chi\overline{\psi}$ is not primitive, but for our purposes the trivial bound $|\tau_n(\chi\overline{\psi})| \leq p^a$ will suffice. We also note that $\tau_0(\chi\overline{\psi}) =0$ as $\chi\overline{\psi}$ is non-trivial here, and we obtain 
\[
 \bigg| \sum_{n \in \Z} \overline{\chi}\psi(n) \, H\left( \frac{n}{y} \right) \bigg|
 \; \leq \;
 y \, \sum_{h \in \Z, \, h \neq 0}   
 \, T\left( \frac{yh}{p^a} \right).
\]
Recall that $y$ is a positive power of $p^a$. We assume that $y$ is chosen so that $y/p^a \rightarrow \infty$ (we will show in the next section that this assumption is consistent with other restrictions on the parameter $y$) and therefore all terms in the sum over $h$ decay exponentially with $p^a$. 
This removes the contribution from the factor $y$, so the right hand side of the above inequality is $O(1)$. 
\par
We deduce that the contribution of the terms with $\chi \neq \psi$ in the sum above is $\ll p^{4ad(1-\Ree(\beta))+2a}$. As we have now assumed $y/p^a \rightarrow \infty$ we see that diagonal terms in the sum are dominant, and we get
\[
\sum_{n \in \Z} |B(n)|^2 H\left( \frac{n}{y} \right) 
\; \ll \; 
p^{4da(1-\Ree(\beta)) + a} \, y .
\]
Recalling the earlier bound for $\Sigma_2$ obtained from Cauchy's inequality, we have
\begin{equation}\label{Sigma2bound}
|\Sigma_2|  \; \ll \;  y^{\Ree(\beta) -1/2+ \epsilon} \; p^{2da(1-\Ree(\beta)) +a/2} .
\end{equation}
\section{Conclusion}\label{Conclusion}
Let us write $p^a=P$ for our parameter which is tending to infinity.
We will now specify $x$ and $y$: we require them to satisfy $xy=P^{2d}$ and to tend to infinity with $P$, so we may write
\[
y \,=\,P^{2d\gamma}  \quad \mbox{and} \quad x \,=\, P^{2d (1-\gamma)}
\] 
for $0<\gamma <1$. Further, we required $y/P$ to tend to infinity, so we assume $\gamma > 1/2d$. 
Putting $\Ree(\beta)=\sigma$, we may write estimates \eqref{Sigma1bound} and \eqref{Sigma2bound} as follows:
\begin{align*}
 |\Sigma_1|   \; & \ll \;   P^{2d(1-\gamma)(3/2-\sigma+\epsilon)} 
\\
\mbox{and} \quad  |\Sigma_2|   \; & \ll \; P^{2d \gamma \, (\sigma-1/2+\epsilon) + 2d(1-\sigma) +1/2} 
\end{align*}
Recall that we have $|A(1)| \gg P$, so $A(1)$ will eventually grow more rapidly than $\Sigma_1$ and $\Sigma_2$ if the following two inequalities hold: 
\begin{align}
 2d(1-\gamma)(3/2-\sigma)   \;&<\; 1,
\label{ineq1}
\\
2d \gamma (\sigma-1/2) + 2d(1-\sigma)  \;&<\; 1/2 .
\label{ineq2}
\end{align}
Observe that we have neglected $\epsilon$ which may be chosen as small as we like. From \eqref{ineq1} and \eqref{ineq2} we deduce the bounds
\[
1-\frac{1}{2d(3/2-\sigma)} \,<\, \gamma  \quad \mbox{and} \quad \gamma  \,<\, \frac{1+4d(\sigma-1)}{4d(\sigma-1/2)} .
\]
One checks that there exists a value of $\gamma$ satisfying both of these inequalities precisely when $\sigma > (6d-1)/(4d+2)$. Further, combining the upper bound on $\gamma$ with the assumption $\gamma > 1/2d$ we obtain $\sigma > 1$; this is a strictly weaker condition, except in the case $d=1$.
Given the functional equation, these are the hypotheses imposed in Theorem \ref{thm1}, and this completes the proof. 
\section*{Acknowledgments}
The author thanks Andrew Booker and Daniel Delbourgo for their helpful suggestions and also thanks the referee for corrections to the paper.
\bibliographystyle{plain}
\bibliography{Ward-Nonvanishing-References}
\end{document}